\documentclass[10pt,a4paper]{article}
\usepackage[utf8]{inputenc}
\usepackage[english]{babel}
\usepackage{amsthm}
\usepackage{amsmath}
\usepackage{amsfonts}
\usepackage{amssymb}
\usepackage{tabu}
\usepackage{caption}
\newtheorem{theorem}{Theorem}[section]

\newtheorem{lem}[theorem]{Lemma}

\newtheorem{rem}[theorem]{Remark}

\author{Taboka Prince Chalebgwa \vspace{0.4cm} \\ \textit{Department of Mathematical Sciences},\\ \textit{Mathematics Division, Stellenbosch University,}\\ \textit{Private Bag X1, 7602 Matieland, South Africa} \\ \textit{Email: taboka@aims.ac.za }}
\title{Algebraic values of certain analytic functions defined by a canonical product}
\date{}
\begin{document}
\maketitle

\begin{abstract}
\noindent We give a partial answer to a question attributed to Chris Miller on algebraic values of certain transcendental functions of order less than one. We obtain $C(\log H)^\eta$ bounds for the number of algebraic points of height at most $H$ on certain subsets of the graphs of such functions. The constant $C$ and exponent $\eta$ depend on certain data associated with the functions and can be effectively computed from them.
\end{abstract}



\section{Introduction}

The current work falls within the general theme of studying the asymptotic density (in terms of height) of algebraic values of bounded height and degree on graphs of transcendental functions. Given $H$ and $d$, a height and a degree bound respectively, a trivial upper bound for the aforementioned density takes the form $C(d) H^{2d}$, which follows immediately from quantitative versions of Northcott's theorem. As such, \textit{polylogarithmic} bounds in $H$ are considered very good, and, for a given transcendental function, often nontrivial to prove.\\

\noindent For the reader's convenience, we begin this section with a brief reminder of the definition of the \textit{absolute multiplicative height} of an algebraic number, which is the height notion we will be using throughout the paper. After this, in order to place our main result within the context of what is known in the general literature, we shall briefly discuss a few related results.\\

\noindent Let $P(z) \in \mathbb{C}[z]$ be a polynomial with complex coefficients. Writing $P(z)$ as

\begin{equation*}
P(z) = a \prod_{j = 1}^n (z - \alpha_j),
\end{equation*}

\noindent the \textit{Mahler measure} $\mathcal{M}(P)$ of the polynomial $P$ is the quantity

\begin{equation*}
\mathcal{M}(P) = |a| \prod_{j = 1}^n \max \{ 1, |\alpha_j| \}.
\end{equation*}

\noindent If $\alpha$ an algebraic number of degree $d$, the \textit{logarithmic} height of $\alpha$, $h (\alpha)$ is defined to be:

\begin{equation*}
h(\alpha) = \frac{\log \mathcal{M} (\alpha)}{d},
\end{equation*}

\noindent where $\mathcal{M}(\alpha)$ is the Mahler measure of the minimal polynomial of $\alpha$ over $\mathbb{Z}$.\\

\noindent The \textit{absolute multiplicative height} of $\alpha$, $H(\alpha)$ is defined as:

\begin{equation*}
H(\alpha) = \exp \left \{  \frac{\log \mathcal{M} (\alpha)}{d}   \right \} = \mathcal{M}(\alpha)^\frac{1}{d}.
\end{equation*}

\noindent If $\alpha$ and $\beta$ are algebraic numbers, we use the notation $H(\alpha, \beta)$ to represent the quantity

\begin{equation*}
\max \{ H(\alpha), H(\beta) \}.
\end{equation*}

\subsection{Some known results}

\noindent In \cite{DaMa}, Masser proves the following result for the number of rational points on the graph of the Riemann $\zeta$-function restricted to the interval $(2,3)$.

\begin{theorem}(Masser, \cite{DaMa})\\
Let $\zeta $ be the restriction of the Riemann $\zeta$-function to the interval $(2,3)$. There is an effective constant $c > 0$ such that for all $H \geq e^e$, the number of rational points of height at most $H$ on the graph of $\zeta $ is at most

\begin{equation*}
c \left(  \frac{\log H}{\log \log H}  \right)^2.
\end{equation*}
\end{theorem}

\noindent In \cite{EtBe1}, adapting Masser's method, Besson studied the density of algebraic points of bounded degree and height on the graph of the $\Gamma$-function restricted to the interval $[n-1,n]$. He obtains the following:

\begin{theorem}(Besson, \cite{EtBe1})\\
There exists a positive effective constant $c$ such that for integers $d \geq 1$, $H \geq 3$ and $n \geq 2$, the number of algebraic points of degree at most $d$ and height at most $H$ on the graph of the $\Gamma$-function restricted to the interval $[n-1,n]$ is at most

\begin{equation*}
c (n^2 \log (n))  \left(  \frac{(d^2 \log H)^2}{\log (d \log H)} \right).
\end{equation*}
\end{theorem}

\noindent In \cite{AnSu}, assuming only that $f$ is complex analytic and transcendental, Surroca achieves the rather exciting bound of $C d^3 (\log H)^2$ for the number of algebraic points of degree at most $d$ and height at most $H$ on the restriction to a compact subset of the graph of $f$. However, the bound is valid only for infinitely many $H$. More precisely:

\begin{theorem}(Surroca, \cite{AnSu})\\
Let $\mathcal{U} \subset \mathbb{C}$ be open and connected. Let $\mathcal{K}$ be a compact subset of $\mathcal{U}$. Let $f$ be a transcendental complex analytic function on $\mathcal{U}$. Then for any integer $d \geq 1$, there exists a positive real number $C > 0$ such that there are infinitely many real numbers $H \geq 1$ such that the number of algebraic points of degree at most $d$ and height at most $H$, with the input belonging to $\mathcal{K}$, is at most

\begin{equation*}
C d^3 (\log H)^2.
\end{equation*}

\end{theorem}

\noindent The constant $C$ effectively depends on $\mathcal{U}$, $\mathcal{K}$ and $f$. It is also shown in the same paper that the theorem cannot be improved any further. That is, one cannot replace the ``infinitely many real $H \geq 1$" in the conclusion of the theorem with ``for all sufficiently large $H$".\\

\noindent Recall that the order and lower order of an entire function $f$ are defined as

\begin{equation*}\label{or1}
\rho = \limsup_{r \rightarrow \infty} \frac{\log \log M(r,f)}{\log r} \  \text{and} \ \lambda = \liminf_{r \rightarrow \infty} \frac{\log \log M(r,f)}{\log r} \ \text{respectively}.
\end{equation*}

\begin{rem}\label{rmm1}
	If $\rho$ is finite, then $\rho$ is the infimum of the set of all $\alpha$ such that $M(r,f) \leq e^{r^\alpha}$ for sufficiently large $r$ and $\lambda$ is the supremum of the set of all $\beta$ such that $e^{r^\beta} \leq M(r,f)$ for sufficiently large $r$.
\end{rem}

\noindent In \cite{BoJo2}, motivated by earlier work of Masser in \cite{DaMa}, Boxall and Jones studied the density of algebraic points of bounded height and degree on graphs of entire functions of finite order $\rho$ and positive lower order $\lambda$ restricted to compact subsets of $\mathbb{C}$. They attain a bound of the form $C(\log H)^\eta$ where the constant $C$ and the exponent $\eta$ are effective and $\eta$ depends only on $\rho$ and $\lambda$. More specifically, they prove the following theorem:\\

\begin{theorem}(Boxall and Jones, \cite{BoJo2})\\
Let $f$ be a nonconstant entire function of order $\rho$ and lower order $\lambda$. Suppose $0 < \lambda \leq \rho < \infty$ and let $d \geq 1$ and $r > 0$. There is a constant $C > 0$ such that for all $H > e$, there are at most $C (\log H)^{\eta (\lambda,\rho)}$ complex numbers $z$ such that $|z| \leq r$, $[\mathbb{Q}(z,f(z)): \mathbb{Q}] \leq d$ and $H(z,f(z)) \leq H$.

\end{theorem}

\noindent The Boxall-Jones theorem immediately prompts two followup questions towards possible generalizations or improvements. On the one hand, one can ask if the same type of bound holds for meromorphic functions. Using Nevanlinna theory, we explored this theme in an upcoming paper currently under preparation. On the other hand, in which cases can the region to which $f$ was initially restricted be enlarged? In fact, more generally, for which functions can one drop the restriction to compact sets and actually count (possibly) \textit{all} points of bounded height and degree on the graph of $f$? In this paper, we explore the second theme for a specific class of entire functions of order less than one, following a question asked by Chris Miller and brought to our attention by Gareth Jones.

\subsection{A proposition of Masser}

A crucial part of our proof strategy involves ``converting" the question of counting algebraic points on the graph of the function $f$ to that of counting (or finding an upper bound for) the number of zeroes of a related function $g$, say, which is a considerably easier to handle task via analytic methods. This requires the construction (or existence) of a non-zero auxiliary polynomial $P(X,Y) \in \mathbb{Z}[X,Y]$ such that $P(z, f(z)) = 0$ whenever

\begin{equation*}
(z, f(z)) \in \overline{\mathbb{Q}}^2, \deg (z, f(z)) \leq d \ \text{and} \  H(z, f(z)) \leq H.
\end{equation*}

\noindent For our purposes, we use the auxiliary polynomial constructed by Masser in Proposition 2 of \cite{DaMa}. We give the details below.
	
	\begin{lem}\label{ml2} (Masser, \cite{DaMa}, Prop. 2)\\
		
\noindent	Let $d \geq 1$ and $T \geq \sqrt{8d}$ be positive integers and $A, Z, M$ and $H$ positive real numbers such that $ H \geq 1$. Let $f_1, f_2$ be functions analytic on an open neighbourhood of $\overline{B(0,2Z)}$, with $\max \{ |f_1(z)|, |f_2(z)| \} \leq M$ on this set. Suppose $\mathcal{Z} \subset \mathbb{C}$ is finite and satisfies the following for all $z, w \in \mathcal{Z}$:
		
		\begin{itemize}
			\item $|z| \leq Z$ ,
			\item $|w-z| \leq \frac{1}{A}$ ,
			\item $[\mathbb{Q}(f_1(z),f_2(z)): \mathbb{Q}] \leq d$,
			\item $H(f_1(z), f_2(z)) \leq H$.
		\end{itemize}
		
\noindent	Then there is a nonzero polynomial $P(X,Y)$ of total degree at most $T$ such that $P(f_1(z), f_2(z)) = 0$ for all $z \in \mathcal{Z}$ provided
		
		\begin{equation*}
		(AZ)^T > (4T)^\frac{96d^2}{T} (M+1)^{16d}H^{48d^2}.
		\end{equation*}
		
\noindent	Moreover, if $|\mathcal{Z}| \geq \frac{T^2}{8d}$, then $P(X,Y)$ can be chosen such that all the coefficients are integers each with absolute value at most
		
		\begin{equation*}
		2^{1/d}(T+1)^2H^T.
		\end{equation*}
		
	\end{lem}

\begin{rem}
When using this lemma, we will take $f_1(z) = z$ and $f_2(z) = f(z)$.
\end{rem}



\section{Preliminaries and auxiliary lemmas}

\subsection{The function $f$, and a brief discussion of the strategy}

Let $1 \leq z_1 \leq z_2 \leq \ldots$ be an increasing and unbounded sequence of positive real numbers such that $\sum_{n=1}^\infty \frac{1}{z_n} < \infty$. Then the infinite product

\begin{equation}\label{eqt}
f(z) := \prod_{n=1}^\infty \left( 1 - \frac{z}{z_n} \right)
\end{equation}

\noindent necessarily defines an entire function of order $\rho$ where $0 \leq \rho <1$.\\

\noindent Chris Miller asked for the density of algebraic points of height at most $H$ and degree at most $d$ on graphs of functions defined in this way. We note that when $f$ has positive lower order $\lambda$, then the Boxall-Jones theorem applies for restrictions of $f$ to sets of the form $\overline{B(0,r)}$ for $r>0$. The bound one gets in this case is of the form $C (\log H)^\eta$ where $C = C(r, f, d, \lambda, \rho)$ and $\eta = \eta (\lambda,\rho)$.\\

\noindent However, as we will see shortly, functions of this form enjoy certain asymptotic approximations that give a more explicit and finer measure of growth than the one provided by just having positive lower order and finite order. Unfortunately, these approximations only hold outside of certain subsets of the graphs. In any case, taking advantage of such explicit growth characterizations, for appropriate subsets of the graphs, one can find the density of \textit{all} the algebraic points of bounded height and degree.\\

\noindent The strategy to do this utilizes a rather simple but crucial observation: Given an algebraic number $z$ of height at most $H$ and degree at most $d$, the modulus $|z|$ is bounded above by a function of $H$ and $d$. Therefore, to count the algebraic points of bounded height and degree on a function $f$, we can restrict our attention to those (algebraic) inputs $z$ for which $|f(z)|$ is not too large to have height at most $H$ or degree at most $d$. This is where an explicit lower approximation of $f$ becomes crucial because it gives us a handle on the growth of $|f|$.\\ 

\noindent The remainder of this subsection is devoted to making the contents of the previous two paragraphs explicit.\\

\noindent Let $0 < \phi < \frac{\pi}{2}$ and denote by $S_\phi$ the sector $S_\phi = \{ z \in \mathbb{C} : -\phi \leq \arg z \leq \phi  \}$. Then the sequence $\{ z_n \}_{n = 1}^\infty \subset S_\phi$. Let the sequence $\{ z_k \}_{k = 1}^\infty$ be such that $1 \leq z_1 \leq z_2 \leq \ldots$ and $\sum_{k = 1}^\infty \frac{1}{|z_k|^p} < \infty$, where $p$ is a non-negative integer. In Example 1 from (\cite{GoOs}, ~pp.66-69), Goldberg and Ostrowski were concerned with approximating the function 
 
\begin{equation*}
g(z) := \prod_{k = 1}^\infty E \left( \frac{z}{z_k} ,p\right)
\end{equation*}

\noindent where $p$ is a non-negative integer and

\[ E(z,p) := \begin{cases} 
      (1-z) & \text{if} \ p = 0 \\
      (1-z) \exp \left( z + \frac{z^2}{2} + \cdots + \frac{z^p}{p}  \right) & \text{otherwise}
   \end{cases}
\]

\noindent is the $p$th \textit{Weierstrass elementary factor}.\\





\subsection*{Lemmas}

\noindent An asymptotic inequality approximating $\log g(z)$ in terms of the function $|z|^\rho$ and certain explicit coefficients was obtained, where $z \in \mathbb{C} \setminus S_\phi $ and $p \leq \rho \leq p+1$. The asymptotic inequality we need is thus a specialization of their result to the case where $p = 0$. We give the specific details in the next lemma.\\

\noindent Let $\{ z_n \}_{n = 1}^\infty$ be the sequence of zeros of $f$ as defined in Equation (\ref{eqt}) and denote by $n(r)$ the number of $z_n$ with modulus less than $r$. Let

\begin{equation*}
\mu := \lim_{r \rightarrow \infty } \frac{n(r)}{r^\rho}
\end{equation*} 
 
\noindent where $\rho \in (0,1) $ is the order of $f$. Assume also that $\lambda > 0$ where $\lambda$ is the lower order of $f$.

\begin{lem}(Corollary of Example 1, \cite{GoOs}, ~pp.66) \label{l1}\\
Let $0 < \epsilon <1$ and suppose $f$, $\mu$, $\rho$ and $\phi$ are as defined previously. Assume $0 < \mu < \infty$. Then there exists $r_1 (\epsilon )$ such that for all $z \in \mathbb{C}$ with $|z| > r_1(\epsilon)$ and $\phi < \arg z < 2 \pi - \phi$,

\begin{equation}\label{et66}
\left | \log f(z) - \frac{\mu \pi}{ \sin \pi \rho } e^{- i \pi \rho} z^\rho   \right | \leq \epsilon A r^\rho \csc \frac{\phi}{2},
\end{equation}

\noindent where $A = 6 + 3 \mu \pi \csc (\pi \rho)$.\\
\end{lem}

\noindent From the above lemma it follows that

\begin{equation*}
\left | \Re \left( \log f(z) - \frac{\mu \pi}{ \sin \pi \rho } e^{- i \pi \rho} z^\rho   \right)  \right | \leq \epsilon A r^\rho \csc \frac{\phi}{2}.
\end{equation*}

\noindent More explicitly, writing $z$ as $z = re^{i\theta}$, where $\phi < \theta < 2\pi - \phi$, we deduce from Inequality (\ref{et66}) that

\begin{equation}\label{eq2}
\left | \log |f(re^{i\theta})| - \frac{\mu \pi}{ \sin \pi \rho } \cos \rho(\theta - \pi) r^\rho   \right | \leq \epsilon A r^\rho \csc \frac{\phi}{2}.
\end{equation}

\noindent Assuming $\rho \in (0, \frac{1}{2}]$, we have on the one hand that  $\rho (\theta - \pi) \in \left( -\frac{\pi}{2} , \frac{\pi}{2}  \right)$. So $\cos \frac{1}{2} (\theta - \pi) \leq \cos \rho (\theta - \pi)$. On the other hand, $\cos \frac{1}{2} (\theta - \pi) = \sin \frac{\theta}{2}$. Therefore, from Inequality (\ref{eq2}), we have that:

\begin{equation*}
|f(re^{i\theta})| \geq e^{ C(\phi , \rho) r^\rho}, \ \ \text{where} \ \ C(\phi , \rho) =   \frac{\mu \pi \sin \frac{\phi}{2} }{ \sin \pi \rho } - \epsilon A \csc \frac{\phi}{2} .
\end{equation*}

\noindent Given $\phi$, $\epsilon$ can be chosen such that

\begin{equation*}
\epsilon = \min \left \{  \frac{\mu \pi \sin^2 \frac{\phi}{2}}{4A \sin \pi \rho }, \frac{1}{2}   \right \}, \text{say}.
\end{equation*}

\noindent In this case, we have that

\begin{equation*}
C (\phi , \rho) > \frac{\mu \pi \sin \frac{\phi}{2} }{2 \sin \pi \rho } > 0.
\end{equation*}

%
%
%
%
%
%
%
%
%

\noindent The next lemma gives us bounds in terms of $H$ and $d$ of the modulus of an algebraic number of height at most $H$ and degree at most $d$. It is essentially a loose version of Liouville's inequality for absolute multiplicative height, which can be found in \cite{MiWa}, pp.82.

%
%

\begin{lem}\label{lmtt}
Let $\alpha$ be a non-zero algebraic number of degree at most $d$ and height at most $H$. Then

\begin{equation}\label{eq3}
\frac{1}{(2H)^d} \leq |\alpha| \leq (2H)^d.
\end{equation}

\end{lem}

\noindent Using the above lemma we prove the following:

\begin{lem}
Let $d \geq 1$ and $H \geq e^e$. Let $z = re^{i \theta} \in \mathbb{C}$ such that $\deg (z) \leq d$ and $H(z) \leq H$. Define the constant $K(\phi , \rho , d)$ by

\begin{equation*}
K(\phi , \rho ,d) = \left( \frac{ 2(d+1) }{C(\phi , \rho)}   \right)^\frac{1}{\rho}.
\end{equation*}

If $r \geq K(\phi , \rho ,d) (\log H)^\frac{1}{\rho} = R_H$, then $e^{ C(\phi , \rho) r^\rho} \geq (2H)^{d+1}$.\\

Hence, for $r \geq \max \{ r_1(\epsilon) , R_H  \}$, we have the following chain of inequalities:

\begin{equation*}
|f(re^{i\theta})| \geq e^{ C(\phi , \rho) r^\rho} \geq (2H)^{d+1}.
\end{equation*}
\end{lem}

\noindent By Lemma \ref{lmtt} this implies that either $H(f(z)) > H$ or $\deg (f(z)) > d$.

\begin{proof}

\noindent We note that $e^{ C(\phi , \rho) r^\rho} \geq (2H)^{d+1}$ if

\begin{equation*}
C(\phi , \rho) r^\rho \geq (d+1) \log (2H).
\end{equation*}

\noindent The above inequality follows if

\begin{equation*}
C(\phi , \rho) r^\rho \geq 2(d+1) \log H.
\end{equation*}

\noindent And this is true if

\begin{equation*}
r \geq K(\phi , \rho ,d) (\log H)^\frac{1}{\rho}.
\end{equation*}

\noindent Recalling that $|f(re^{i\theta})| \geq e^{ C(\phi , \rho) r^\rho}$ when $r \geq r(\epsilon)$, if $r \geq \max \{ r_1(\epsilon) , R_H  \}$, we obtain the desired chain of inequalities.

\end{proof}

\noindent The next lemma gives a quantitative way of covering the zeroes of a polynomial $P(z)$ with a collection of disks outside of which $|P(z)| > 1$.

\begin{lem}(Boutroux-Cartan)\label{l54}\\
Let $P(z) \in \mathbb{C}[z]$ be a monic polynomial with degree $n \geq 1$. Then $|P(z)| > 1$ for all complex $z$ outside a collection of at most $n$ disks the sum of whose radii is $2e$.
\end{lem}

\noindent In the following lemma, the function $n(r,\frac{1}{f})$ represents the number of zeroes of $f$ in $\overline{B(0,r)}$. This is a standard Nevanlinna theoretic notation.

\begin{lem}(A corollary of Jensen's formula)\\
Let $G$ be a nonconstant entire function such that $G(0) \neq 0$. Let $0 < r < R < \infty$. Then:

\begin{equation*}
n(r,\frac{1}{G}) \leq \frac{1}{\log \frac{R}{r}} \log \left(  \frac{M(R,G)}{|G(0)|}  \right)
\end{equation*}

\end{lem}

\section{Main Result}

\noindent We can now state and prove the main result of this section. Since $f$ is an entire function of positive lower order and finite order, our argument is an adaptation of that of Boxall and Jones in \cite{BoJo2}.\\

\begin{theorem}
Let $f (z) = \prod_{n = 1}^\infty \left( 1 - \frac{z}{z_n} \right)$ where $1 \leq z_1 \leq z_2 \leq \ldots$ and $\sum_{n=1}^\infty \frac{1}{z_n} < \infty$. Suppose the lower order $\lambda$ and order $\rho$ of $f$ are such that $0 < \lambda \leq \rho \leq \frac{1}{2}$. Let $0 < \phi < \frac{\pi}{2}$. Let $d, \alpha, \beta, \gamma$ be as follows: $d \geq 1$, $ \alpha = 1 + \rho$, $\beta = \frac{\lambda}{2}$, and  $\gamma = \frac{2\alpha + \rho}{\beta \rho}$. Then there is a constant $C > 0$ such that for all $H>e$, there are at most $C (\log H)^{\frac{2\alpha (\gamma + 1)}{\rho}}$ numbers $z \in \mathbb{C} \setminus S_\phi$ such that $[\mathbb{Q}(z, f(z)): \mathbb{Q}]\leq d$, $H(z, f(z)) \leq H$.
\end{theorem}

\begin{proof}
Let $H >e^e$. Throughout our proof the height bound $H$ is assumed to be sufficiently large. We shall denote by $C$ a positive constant independent of $H$. The constant $C$ may not be the same at each occurrence. Recall that $|P|$ denotes the modulus of the coefficient of the polynomial $P$ with largest absolute value.\\

\noindent We would first like to obtain a non-zero polynomial $P(X,Y) \in \mathbb{Z}[X,Y]$ of degree at most $T = C (\log H)^\frac{2\alpha}{\rho}$ such that $|P| \leq 2^\frac{1}{d} (T+1)^2 H^T$ and $P(z, f(z)) = 0$ whenever $[\mathbb{Q}(z, f(z)): \mathbb{Q}]\leq d$, $H(z, f(z)) \leq H$ and $z \not \in S_\phi$. To this end, let:

\begin{equation*}
A = \frac{1}{2R_H}, \ \ Z = C (\log H)^\frac{1}{\rho}, \ \ T = C (\log H)^\frac{2\alpha}{\rho} \ \ \text{and} \ \ M = e^{(2Z)^\alpha}.
\end{equation*}

\noindent We then have that $\max \{|z|, |f(z)| \} \leq M$ for all $z \in \overline{B(0, 2Z)}$.\\

\noindent Furthermore, we note that

\begin{equation*}
\log (AZ)^T  = C (\log H)^\frac{2\alpha}{\rho}  > C \left( \frac{\log \log H}{(\log H)^\frac{2\alpha}{\rho}} \right) + C (\log H)^\frac{\alpha}{\rho} + C\log H.
\end{equation*} 

\noindent Therefore

\begin{equation*}
(AZ)^T > (4T)^\frac{96d^2}{T} (M+1)^{16d}H^{48d^2}.
\end{equation*}
 
\noindent We note that the bound we are trying to prove is worse than $C (\log H)^\frac{4\alpha}{\rho}$. We can thus assume that there are at least $\frac{T^2}{8d}$ complex numbers such that $[\mathbb{Q}(z, f(z)): \mathbb{Q}]\leq d$ and $H(z, f(z)) \leq H$. By Lemma \ref{ml2} there is a polynomial $P(X,Y)$ satisfying all our requirements.\\

\noindent Let $G(z) = P(z, f(z))$. We would like to bound the number of zeroes of $G$ in $\overline{B(0, R_H)}$. To do this, first let $k$ be the highest power of $Y$ in $P(X,Y)$. We can assume $k \geq 1$. Let $\tilde{P}(X,Y) = Y^k P(X, \frac{1}{Y})$, $R(X) = \tilde{P}(X,0)$, and $Q(X,Y) = \tilde{P}(X,Y) - R(X)$. We note that $R(X)$ is not identically zero.\\

\noindent Let $\tilde{Q}(X,Y) = \frac{1}{Y}Q(X,Y)$. The highest power of $X$ in $\tilde{Q}$ is at most $T$ and $|\tilde{Q}| \leq |P| \leq 2^\frac{1}{d} (T+1)^2 H^T$. Finally, $\tilde{Q}$ has at most $(T+1)^2$ terms.\\

\noindent We would now like to find some $z_i \in \mathbb{C}$ such that $|G(z_i)| = |P(z_i, f(z_i))| \geq 1$. To this end, first we would like to find some sufficiently large radius $r$ such that if $|z| \geq r$, then $\left| Q \left(z, \frac{1}{f(z)}   \right) \right| \leq \frac{1}{2}$.\\

\noindent Let $z = re^{i \theta} \in \mathbb{C}$ be such that $|f(z)| = M(r, f) \geq 1$. Then:

\begin{equation*}
\left| \tilde{Q} \left(z, \frac{1}{f(z)}   \right)  \right| \leq 2^\frac{1}{d} (T+1)^4 H^T r^T.
\end{equation*}

\noindent Therefore

\begin{equation*}
\left| Q \left(z, \frac{1}{f(z)}   \right)  \right| \leq \frac{1}{2}
\end{equation*}

\noindent provided

\begin{equation*}
2^\frac{1}{d} (T+1)^4 H^T r^T \leq \frac{1}{2} M(r,f).
\end{equation*}

\noindent We note that (for a large enough $C$) if

\begin{equation*}
r \geq C (\log H)^{(2\alpha + \rho) / \beta \rho},
\end{equation*}

\noindent then

\begin{equation*}
2^\frac{1}{d} (T+1)^4 H^T r^T \leq \frac{1}{2} e^{r^\beta}
\end{equation*}
 
\noindent and, by Remark \ref{rmm1}

\begin{equation*}
e^{r^\beta} \leq M(r,f).
\end{equation*}

\noindent We thus get that

\begin{equation*}
\left| Q \left(z, \frac{1}{f(z)}   \right)  \right| \leq \frac{1}{2}
\end{equation*}

\noindent when 

\begin{equation*}
r \geq C (\log H)^\frac{(2\alpha + \rho)}{\beta \rho}.
\end{equation*}

\noindent Note that the degree of $R(X)$ is also at most $T$. For $i = 1, \ldots , [T] + 14$, say, let $r_i$ be the $i$th integer after $C (\log H)^\gamma$ where $\gamma := \frac{2\alpha + \rho}{\beta \rho}$. Let $z_i$ be such that $|z_i| = r_i$ and $|f(z_i)| = M(r_i,f)$. By Lemma \ref{l54}, there will be at least one $i$ such that $|R(z_i)|>1$. For such $i$, we have:

\begin{equation*}
\left| \tilde{P} \left(z_i, \frac{1}{f(z_i)}   \right)  \right| \geq \frac{1}{2}.
\end{equation*}

\noindent We can (again by Remark \ref{rmm1}) conclude that

\begin{equation*}
|G(z_i)| = |P(z_i, f(z_i))| = \left| f(z_i)^k \tilde{P} \left( z_i, \frac{1}{f(z_i)}  \right)  \right| \geq \frac{1}{2}e^{k{r_i}^\beta},
\end{equation*}

\noindent and therefore

\begin{equation*}
|G(z_i)| \geq 1.
\end{equation*}

\noindent Recall that $R_H$ is of the form $C (\log H)^\frac{1}{\rho}$ whilst on the other hand, \newline $r_i \leq C (\log H)^\gamma +T + 14$. So, $\overline{B(0, R_H)} \subset \overline{B(z_i, s)}$ where $s = C (\log H)^\gamma$.\\

\noindent By the maximum modulus principle and Lemma \ref{ml2}, we have that

\begin{equation*}
n(R_H, \frac{1}{G}) \leq \frac{1}{\log 2} \log \left( \frac{M(3s, G)}{|G(z_i)|}  \right) \leq \frac{\log M(3s, G)}{\log 2}.
\end{equation*}

\noindent By Remark \ref{rmm1}, we have that

\begin{equation*}
M(3s, G) \leq |P| (T+1)^2 (3s)^T e^{T(3s)^\alpha}.
\end{equation*}

\noindent Since $s = C (\log H)^\gamma$ and $T = C (\log H)^\frac{2\alpha}{\rho}$, we deduce that

\begin{equation*}
\log M(3s, G) \leq C (\log H)^{\frac{2\alpha (\gamma + 1)}{\rho}}.
\end{equation*}

%

\noindent Therefore:

\begin{equation*}
n(R_H, \frac{1}{G}) \leq C (\log H)^{\frac{2\alpha (\gamma + 1)}{\rho}}.
\end{equation*}

\noindent as required.\\

\noindent The constant $C$ effectively depends on $\mu, \lambda, \rho, \phi$ and $d$.

\end{proof}

\begin{center}
\textbf{\large{Acknowledgements}}
\end{center}

\noindent I would like to thank Gareth Boxall and Gareth Jones for bringing this question to my attention, as well as for the subsequent insightful comments and suggestions. The contents of this paper formed a chapter in the author's PhD thesis, which was completed under the guidance of Gareth Boxall.

\end{document}